\documentclass{article}

\usepackage{hyperref}
\usepackage[english]{babel}

\usepackage{graphicx}
\usepackage{amssymb}
\usepackage{amsthm, amsmath,amssymb}
\usepackage[all]{xy}
\usepackage{pstricks,pstricks-add,pst-node,pst-text,pst-3d}
\usepackage{graphicx}
\usepackage{psfrag}
\usepackage{multirow}

\usepackage[applemac]{inputenc}

\def\Rb{\mathbb{R}}		

\newcommand{\rk}{\mathit{rank }}

\newcommand{\pder}[2]{\frac{\partial #1}{\partial #2}}

\newcommand{\restr}[1]
   {\vrule height1ex width.4pt
depth1.4ex\lower1.4ex\hbox{\scriptsize $\,#1$}}

\newtheorem{theorem}{Theorem}

\newtheorem{proposition}{Proposition}

\newtheorem{lemma}{Lemma}

\theoremstyle{definition}
\newtheorem{definition}{Definition}
\newtheorem{rem}{Remark}
\newtheorem{example}{Example}

\author{Inês Cruz\thanks{Centro de Matem\'atica da Universidade do Porto, Departamento de Matem\'atica, Faculdade de Ci\^encias da Universidade do Porto, 4169-007 Porto, Portugal}\,  and M. Esmeralda Sousa-Dias\thanks{Departamento de Matemática, Center for Mathematical Analysis, Geometry, and Dynamical Systems (CAMGSD-LARSyS),  Instituto Superior Técnico, 1049-001 Lisboa, Portugal.}}

\begin{document}
\title{Reduction of cluster iteration maps to symplectic maps}
\maketitle

  \begin{abstract}
We study iteration maps of recurrence relations arising from mutation periodic quivers of arbitrary period. Combining tools from cluster algebra theory and (pre)symplectic geometry, we show that these cluster iteration maps can be reduced to symplectic maps on a lower dimensional submanifold, provided  the matrix representing the quiver is singular. The reduced iteration map is explicitly computed for several new periodic quivers. 
\end{abstract}
 
\medskip

\noindent {\it MSC 2010:} 53D20; 13F60, 37J10, 65Q30.\\
{\it Keywords:}  symplectic reduction, cluster algebras, symplectic maps, recurrence relations.

\section{Introduction} 

Recurrence relations arise in a natural way from periodic quivers via Fomin-Zelevinsky cluster mutations. We call  this type of relations  {\it cluster recurrence relations}. The iteration map of such recurrence relations is birrational and defined as  the composition of mutations and permutations.
In  \cite{FoMa} the  notion of mutation-periodicity of a quiver is used to show that 1-periodic quivers give rise to recurrence relations on the real line while  to higher periodic quivers correspond  recurrence relations on higher dimensional spaces.

Since the introduction of cluster algebras  by Sergey Fomin and Andrei Zelevinsky in \cite{FoZe}, the theory of cluster algebras has grown  in many research directions. Relevant to our work are the relations between cluster algebras and Poisson/symplectic geometry whose main achievements are surveyed in \cite{GeShVa3}. The presymplectic structures considered in the context of cluster algebra theory are known as log-canonical presymplectic structures. We will refer to these forms just as {\it log presymplectic forms}. 

Our starting motivation was to understand the relevance of  presym\-plectic and Poisson structures  compatible with a cluster algebra to the recurrence relations arising from such algebras.  Some steps in this direction were taken in \cite{FoHo} and \cite{FoHo2}, where the  integrability of some Somos-type sequences (associated to 1-periodic quivers) was obtained by reduction of the iteration map to a symplectic map. These references, based on the classification of 1-periodic quivers   obtained in \cite{FoMa}, prove that iteration maps arising from a 1-periodic quiver can be reduced to  symplectic maps.

Our main result, Theorem~\ref{reducth}, shows that any iteration map arising from a quiver of arbitrary period $m$ can be reduced to a symplectic map on a $2k$-dimensional space with respect to a log symplectic form, where $2k$ is the rank of the matrix representing the quiver. The proof of this theorem does not rely on the classification of  periodic quivers, which  is unknown for periods higher than 1.  The main ingredients of the proof are: (i) the invariance of the standard log presymplectic form under the iteration map, which is proved in Theorem~\ref{presomega}; (ii) a classical theorem of G. Darboux (or of E. Cartan for the linear version) for the reduction of an arbitrary presymplectic form to a symplectic form.

To illustrate the main result, we explicitly compute the  reduced symplectic iteration map of  some cluster recurrence relations. We would like to point out  the relevance of the constructive proof of E. Cartan's theorem to the computations carried out. In fact, this construction  provides explicit  Darboux coordinates necessary to the computation of the reduced iteration map.

We note that iteration maps are not intrinsic objects of study in cluster algebra theory since they only appear when one considers cluster algebras associated to periodic quivers. Several   results presented here  provide  new insights into the role of some cluster algebras structures,  in particular of the so-called secondary cluster manifold. In fact,   Theorem~\ref{reducth} can be interpreted as reduction of an extra structure - the iteration map - to the secondary cluster manifold with its Weil-Petersson form (see  \cite{GeShVa2}, \cite{GeShVa} and \cite{GeShVa3}). 

The organization of the paper is as follows. In Section~\ref{sec2} we introduce the basic notions of the theory of cluster algebras necessary to subsequent sections, in particular, the definition of periodic quivers and the construction of cluster recurrence relations from  periodic quivers.  The next section is devoted to the proof of Theorem~\ref{presomega} which shows the invariance, under the iteration map, of  the log presymplectic form   whose coefficient matrix is (up to a constant) the matrix representing the periodic quiver.  In Section~\ref{sec4}, we prove the main result, Theorem~\ref{reducth}, on the reduction of cluster iteration maps to symplectic maps. The paper ends   with  examples illustrating the previous results.

\section{Mu\-ta\-tion-periodic quivers and cluster iteration maps}\label{sec2}

Here we introduce the notions of the theory of cluster algebras  necessary to the following  sections. We will work in the context of coefficient free cluster algebras ${\cal A}(B)$ where $B$ is a (finite) skew-symmetric integer matrix. 

In this work {\it quiver} means an oriented graph with $N$ nodes and no loops nor 2-cycles. It will be represented by an $N$-sided polygon whose vertices are the nodes of the quiver and will be labelled by  $1,2,\ldots, N$ in clockwise direction.  To each  oriented edge of the polygon one associates a weight which is the positive integer representing the number of arrows between the corresponding nodes of the quiver. A quiver can also be identified with a skew-symmetric  matrix $B=[b_{ij}]$, being  $b_{ij}$  the number of arrows from node $i$ to node $j$ minus the number of arrows from $j$ to $i$. We denote by $B_Q$ the $N\times N$ skew-symmetric matrix  representing a quiver $Q$ with $N$ nodes.

To each node $i$ of a quiver $Q$ one attaches a variable $u_i$  called  {\it cluster variable}.  The pair $(B_Q,\mathbf{u})$ is called the {\it initial seed} and $\mathbf{u}=(u_1,\ldots,u_N)$  the {\it initial cluster}. 

The basic operation of the theory of cluster algebras  is called a {\it mutation}.  A mutation $\mu_k$ in {\it the direction} of $k$ (or at node $k$) acts  on a given seed $(B,\mathbf{u})$ with $B=[b_{ij}]$ and $\mathbf{u}=(u_1, \ldots, u_N)$, as follows:
\begin{itemize}
\item $\mu_k(B)=\left[b'_{ij}\right]$ with
\begin{equation}\label{mut k2}
b'_{ij}=\begin{cases}
-b_{ij},&\text{$ (k-i)(j-k)=0$}\\
b_{ij}+\frac{1}{2} \left(|b_{ik}|b_{kj}+ b_{ik}|b_{kj}|\right),&\text{otherwise}.
\end{cases}
\end{equation}
\item $\mu_k(u_1, \ldots, u_N)=  (u'_1,\ldots, u'_N)$, with
\begin{equation}\label{mut k1}
u'_i=\begin{cases}
u_i& i\neq k\\
\displaystyle{\frac{\prod_{j: b_{kj}>0}u_j^{b_{kj}}+\prod_{j: b_{kj}<0}u_j^{-b_{kj}}}{u_k},}&i=k.
\end{cases}
\end{equation}
When  one of the products in \eqref{mut k1} is taken over an empty set, its value is assumed to be 1.
\end{itemize}

 Formulae  \eqref{mut k2} and \eqref{mut k1} are called (cluster) {\it exchange relations}. It is easy to see that the exchange relations have the properties:  (i) if $B$ is skew-symmetric, then  $B'=\mu_k(B)$ is again skew-symmetric; (ii) $\mu_k$ is an involution, that is $\mu_k\circ\mu_k=Id$.

Given a matrix $B$ and an initial cluster $\mathbf{u}$ we can apply a mutation $\mu_k$ to produce another cluster, and then apply another mutation $\mu_p$ to this cluster to produce another cluster and so on.  The {\it cluster algebra} (of geometric type) ${\cal A}(B)$ is the subalgebra of the field of rational functions in the cluster variables, generated by the union of all clusters.

\bigbreak

The  notion of  {\it mutation periodicity} of a quiver introduced in \cite{FoMa} enables  to associate a recurrence relation to a periodic quiver.  More precisely,  to an $m$-periodic quiver one associates a recurrence relation on an $m$-dimensional space. Let us recall the definition of periodic quiver. 

Consider the permutation $\sigma: ( 1, 2, \ldots, N ) \longmapsto  (2,3,\ldots, N,1)$ and 
 $\sigma Q$  the quiver in which the number of arrows from node $\sigma(i)$ to node $\sigma(j)$ is the number of arrows in $Q$ from node $i$ to node $j$. Equivalently,  the action of $\sigma$ on the polygon representing $Q$ leaves the weighted  edges  fixed and moves the vertices in the counterclockwise direction.
 
  If $B_Q$ is the skew-symmetric matrix representing the quiver $Q$, the action $Q\mapsto \sigma Q$ corresponds to the conjugation $B_Q\mapsto \sigma^{-1} B_Q\, \sigma $, that is,
$$B_{\sigma Q} = \sigma^{-1} B_Q\, \sigma,$$ 
where, slightly abusing notation, $\sigma$ also denotes the matrix representing the permutation, that is
\begin{equation}
\label{sigma}
\sigma = \begin{bmatrix}
0&1&0&\cdots&0\\
0&0&1&&\vdots\\
\vdots&&\ddots&\ddots&\vdots\\
&&&&1\\
1&0&0&\cdots&0
\end{bmatrix} .
\end{equation}
 
 \begin{definition} Let $Q$ be a quiver with $N$ nodes, $B_Q$ the skew-symmetric matrix representing $Q$  and  
$$\sigma: ( 1, 2, \ldots, N ) \longmapsto  (2,3,\ldots, N,1).$$
$Q$ is said to have {\it period} $m$ if   $m$  is the smallest positive integer such that
\begin{equation}\label{matrixper}
\mu_m\circ\cdots\circ\mu_1(B_Q)=\sigma^{-m} B_Q\, \sigma^m.
\end{equation}
\end{definition}

\bigbreak

The  notion of mutation periodicity means that if a quiver is  $m$-periodic, then after applying $\mu_m\circ\cdots\circ\mu_1$  we return to a quiver which is equivalent (up to a certain permutation) to the original quiver, and so mutating this quiver at node $m+1$ produces an exchange relation identical in  form to the exchange relation at node 1 but with a different labeling. The next mutation at node $m+2$ will produce an exchange relation whose form is identical to the exchange relation at node 2 with a different labeling, and so on. This process produces a list of exchange relations, which is interpreted as a recurrence relation.

We note that the definition of a periodic quiver could also be stated in terms of the action of a mutation on a quiver which is defined by specifying a set of rules for mutating the arrows  of the quiver (see for instance \cite{FoMa} and \cite{Ke}).

\bigbreak

We now  explain with a running example how to construct the  recurrence relation (and the  respective iteration map)  corresponding  to a periodic quiver. 

\begin{example}
Consider an initial cluster $\mathbf{u}=(u_1,u_2,u_3,u_4,u_5,u_6)$
and the quiver with 6 nodes represented by the matrix 

\begin{equation}\label{matrix62per}
B=\begin{bmatrix}
0&-r&s&-p&s&-t\\
r&0& -t-rs&s&-p-r s&s\\
-s&  t+rs&0& -r-s(t-p)&s&-p\\
p&-s& r+s(t-p)&0& -t-r s&s\\
-s&p+r s&-s& t+r s&0&-r\\
t&-s&p&-s&r&0
\end{bmatrix}
\end{equation}
where $r,s,t,p$ are positive integers.
Using \eqref{mut k2} and \eqref{sigma} it is easy to check that  one has  $\sigma^{-1} B\sigma=\mu_1(B)$ if $r=t$ and  $\sigma^{-2} B\sigma^2=\mu_2\circ \mu_1(B)$ if $r\neq t$. That is, the quiver is 1-periodic when $r= t$ and 2-periodic otherwise.

\begin{itemize}
\item[1.] {\bf Case $r=t$}

According to \eqref{mut k1}, the mutation at node 1 produces the cluster $\mu_1(\mathbf{u})= (u_7,u_2,u_3,u_4,u_5,u_6)$  with
\begin{equation}\label{1p6nodes}
u_7 u_1=u_2^ru_4^pu_6^r+u_3^su_5^s,
\end{equation}
 and $u_7=u'_1$. The exponents of the right hand side monomials were read directly from the first row of $B$. As the quiver is 1-periodic, mutating now at node 2 produces an exchange relation which is a \emph{shift by 1 of the relation} \eqref{1p6nodes}. Indeed, $\mu_2\circ\mu_1(\mathbf{u}) = (u_7,u_8,u_3,u_4,u_5,u_6)$ with 
$$u_8u_2= u_3^ru_5^pu_7^r+u_4^su_6^s,
$$
where the exponents of the right hand side monomials can now be read from the second row of $\mu_1(B)$ which is $(-r,0,-r,s,-p,s)$. Mutating successively at consecutive nodes we obtain the following sixth order  cluster recurrence relation on the real line:
\begin{equation}\label{rec6nodes1p}
u_{n+6} u_n= u_{n+1}^ru_{n+3}^ru_{n+5}^p+u_{n+2}^su_{n+4}^s, \quad n=1,2,\ldots
\end{equation}
\item[2.] {\bf Case $r\neq t$}

The quiver represented by $B$ is now 2-periodic. Mutating the initial cluster at node 1 produces the cluster $\mu_1(\mathbf{u})= (u_7,u_2,u_3,u_4,u_5,u_6)$ with 
\begin{equation}\label{2p6nodes1}
u_7 u_1=u_2^ru_4^pu_6^t+u_3^su_5^s.
\end{equation}
The second row of $\mu_1(B)$ is now $(-r,0,-t,s,-p,s)$, and so the cluster $\mu_2\circ\mu_1(\mathbf{u})= (u_7,u_8,u_3,u_4,u_5,u_6)$ satisfies the relation
\begin{equation}\label{2p6nodes2}
u_8u_2= u_3^tu_5^pu_7^r+u_4^su_6^s,
\end{equation}
with $u_7$ given by \eqref{2p6nodes1}.
Unlike the previous case, the exchange relation \eqref{2p6nodes2} is not  a shift of the relation \eqref{2p6nodes1}. However, as the quiver is 2-periodic, mutating at node 3 gives an exchange relation which is a shift by 2 of  \eqref{2p6nodes1} and the next mutation at node 4 gives an exchange relation which is a shift by 2 of  \eqref{2p6nodes2}. That is, the 2-periodic quiver gives rise to the following third order   cluster recurrence relation  on the plane:
\begin{equation}\label{rec6nodes}
\left\{\begin{array}{ll}
x_{n+3}x_n&=y_{n}^ry_{n+1}^py_{n+2}^t+x_{n+1}^sx_{n+2}^s\\
&\\
y_{n+3}y_n&=x_{n+1}^t x_{n+2}^px_{n+3}^r+y_{n+1}^sy_{n+2}^s,
\end{array}n=1,2,\ldots
\right.
\end{equation}
where $x_n=u_{2n-1}$ and $y_n=u_{2n}$.
\end{itemize} 

\end{example}

\begin{rem} The form of the matrices representing  1-periodic quivers  was obtained  in \cite{FoMa}, however the classification of quivers of higher period is still unknown.  To the best of our knowledge the  2-periodic quiver  with 6 nodes represented by the matrix $B$ in \eqref{matrix62per}  is new.
\end{rem}
From the construction of the recurrence relation associated to an $m$-periodic quiver, it is easy to see that the corresponding {\it cluster iteration map} is given by
\begin{equation}\label{itmapm}
\varphi= \sigma^m\circ\mu_m\circ\cdots\circ\mu_2\circ\mu_1.
\end{equation}
In particular, the iteration maps for the recurrence relations \eqref{rec6nodes1p} and \eqref{rec6nodes} are, respectively, 
\begin{align}\label{itmap4nos2pernova1}
\varphi(u_1,u_2,u_3,u_4,u_5,u_6)&=\left ( u_2,u_3,u_4,u_5,u_6,\underbrace{\frac{u_2^ru_4^pu_6^r+u_3^su_5^s}{u_1}}_{u_7} \right ),\\ \label{itmap4nos2pernova2}
\varphi(u_1,u_2,u_3,u_4,u_5,u_6)&=\left ( u_3,u_4,u_5,u_6,\underbrace{\frac{u_2^ru_4^pu_6^t+u_3^su_5^s}{u_1}}_{u_7},\underbrace{\frac{u_3^tu_5^pu_7^r+u_4^su_6^s}{u_2}}_{u_8}\right ).
\end{align}

\section{Log presymplectic forms and  periodic quivers}\label{sec3}

 Presymplectic structures (and quadratic Poisson structures) associated to a cluster algebra ${\cal A}(B)$ were introduced in   \cite{GeShVa2} and \cite{GeShVa}. The  presymplectic forms considered in the context of cluster algebras are  of the type $\omega=\sum_{i<j} w_{ij} \frac{du_i}{u_i}\wedge \frac{du_j}{u_j}$. We will call these forms  {\it log presymplectic  forms} and the functions $u_i$ {\it log coordinates} with respect to $\omega$. Such   log presymplectic form is said to be compatible with a cluster algebra ${\cal A}(B)$ if all the clusters in ${\cal A}(B)$ give log coordinates with respect to $\omega$,  with eventually different coefficients $\omega'_{ij}$. It was shown in  \cite{GeShVa} that any cluster algebra carries a compatible log presymplectic structure, and whenever the matrix $W=[w_{ij} ]$ has not full rank  there exits a rational (symplectic) manifold of dimension $2k=\operatorname{rank} W$ called the secondary cluster manifold. The symplectic form on this manifold is called the Weil-Petersson form associated to the cluster algebra ${\cal A}(B)$. The reason for this name is its relation with the Weil-Petersson form on a Teichm\"uller space (see  \cite{GeShVa} or \cite{GeShVa3} for details).

 Our main aim is to understand  the relevance of these log presymplectic structures to the recurrence relations arising from $m$-periodic quivers. As we will show in the next section, if the matrix representing the quiver is singular the recurrence's iteration map can be reduced to a symplectic map with respect to a log symplectic form.
 The key property behind this symplectic reduction is precisely the invariance of the standard log presymplectic structure \eqref{ssf} under the iteration  map.

 \bigbreak

If  $(B,\mathbf{u})$ is  the initial seed with  $B=[b_{ij}]$, we call the 2-form 
\begin{equation}
\label{ssf}
\omega = \sum_{1\leq i<j\leq N} b_{ij} \frac{du_i}{u_i}\wedge \frac{du_j}{u_j},
\end{equation}
the {\it standard log presymplectic form} associated to the cluster algebra ${\cal A}(B)$. The skew-symmetric matrix $B$ will be called the {\it coefficient matrix of $\omega$}.

The 2-form $\omega$ is known in the literature as log-canonical presymplectic form, since in the coordinates $v_i= \log u_i$  it has the following {\it canonical} form:
$$\omega = \sum_{i<j} b_{ij} dv_i\wedge dv_j.$$

Although we do not use explicitly  the  notion of compatibility of  a presymplectic form with a cluster algebra, we remark that  the standard log presymplectic form \eqref{ssf} is in fact compatible with the cluster algebra ${\cal A}(B)$ (see for instance Theorem 6.2 in  \cite{GeShVa3} which characterizes  such compatible forms). 

\begin{theorem}
\label{presomega}
Let $(B,\mathbf{u})$ be an initial seed,  and  $\omega$ the standard log presymplectic form \eqref{ssf} associated to ${\cal A}(B)$. Then, the following are equivalent:
\begin{enumerate}
\item The matrix $B$ represents  an m-periodic quiver, that is
$$\mu_m \circ \mu_{m-1}\circ \cdots \circ \mu_1 (B) = \sigma^{-m} B\sigma^m.$$
\item  $\varphi^* \omega=\omega$, where $\varphi = \sigma^m \circ \mu_m \circ \mu_{m-1}\circ \cdots \circ \mu_1$.
\end{enumerate}
\end{theorem}

The proof of this theorem relies on the following two lemmas.

\begin{lemma}
\label{lem1}
Let $(B,\mathbf{u})$ be an initial seed, $\omega$ the standard log presymplectic form \eqref{ssf} and $\sigma$  the permutation (\ref{sigma}). Then, the \emph{pullback of $\omega$ by $\sigma$} is given by
\begin{equation}\label{pullsigma}
\sigma^* \omega = \sum_{i<j} \left ( \sigma^{-1} B\sigma \right )_{ij} \frac{du_i}{u_i}\wedge \frac{du_j}{u_j}.
\end{equation}
\end{lemma}

\begin{proof} As $\sigma(u_1, u_2,\ldots, u_N) = (u_2, \dots, u_N, u_1)$, the pullback of $\omega$ by $\sigma$ is given by 
\begin{eqnarray}
\sigma^* \omega & = & \sum_{1\leq i\leq N-1} b_{iN} \frac{du_{i+1}}{u_{i+1}}\wedge \frac{du_{1}}{u_{1}} + \sum_{1\leq i<j\leq N-1} b_{ij} \frac{du_{i+1}}{u_{i+1}}\wedge \frac{du_{j+1}}{u_{j+1}}  \nonumber \\ 
& = & -\sum_{2\leq k\leq N} b_{k-1,N} \frac{du_1}{u_1}\wedge \frac{du_k}{u_k} + \sum_{2\leq k<l\leq N} b_{k-1,l-1} \frac{du_k}{u_k}\wedge \frac{du_l}{u_l}.
\end{eqnarray}
Therefore the coefficient matrix of $\sigma^* \omega$ is
\begin{equation}
\label{b'}
\hat{B}=\begin{bmatrix}
0 & -b_{1,N} & -b_{2,N} & \cdots & -b_{N-1,N}\\
b_{1,N} & 0 & b_{1,2} & \cdots & b_{1,N-1}\\
\vdots & \vdots & \vdots & \vdots & \vdots \\
b_{N-2,N} & -b_{1,N-2} & -b_{2,N-2} & \cdots & b_{N-2,N-1}\\
b_{N-1,N} & -b_{1,N-1} & -b_{2,N-2} & \cdots & 0
\end{bmatrix}.
\end{equation}
A straightforward computation shows that $\sigma^T B = \hat{B} \sigma^T$. As $\sigma$ is orthogonal, this is equivalent to $\sigma^{-1}B\sigma = \hat{B}$, which concludes the proof.
\end{proof}

\begin{lemma}
\label{lem2}
 Let $(B,\mathbf{u})$ be an initial seed, $\omega$ the standard log presymplectic form \eqref{ssf} and $\mu_k$ the mutation in the direction $k$ given by \eqref{mut k2} and \eqref{mut k1}. Then, the \emph{pullback of $\omega$ by $\mu_k$} is given by
\begin{equation}\label{pullmuk}
\mu_k^* \omega = \sum_{i<j} \left ( \mu_k (B) \right )_{ij} \frac{du_i}{u_i}\wedge \frac{du_j}{u_j}.
\end{equation}

\end{lemma}

\begin{proof} Recall from \eqref{mut k1} that
$$\mu_k(u_1,u_2, \dots , u_N)= \left ( u_1, u_2, \ldots , u_{k-1}, \underbrace{\frac{A^+ + A^-}{u_k}}_{u'_k}, u_{k+1}, \ldots, u_N\right ),$$
with 
$$A^+=\prod_{l: b_{kl}>0}u_l^{b_{kl}}, \quad A^-=\prod_{l: b_{kl}<0}u_l^{-b_{kl}}.$$
The pullback of $\omega$ by $\mu_k$ is then given by
\begin{equation}
\label{calc}
\mu_k^* \omega  =  \sum_{ k\neq i<j\neq k} b_{ij} \frac{du_i}{u_i}\wedge \frac{du_j}{u_j} + \sum_{j>k} b_{kj} \frac{du'_k}{u'_k}\wedge \frac{du_j}{u_j} +  \sum_{i<k} b_{ik} \frac{du_i}{u_i}\wedge \frac{du'_k}{u'_k}.
\end{equation}
As
$$\frac{du'_k}{u'_k}= - \frac{du_k}{u_k} + \frac{A^+}{A^++A^-}\sum_{l: b_{kl}>0} b_{kl}\frac{du_l}{u_l} - \frac{A^-}{A^++A^-}\sum_{l: b_{kl}<0} b_{kl}\frac{du_l}{u_l},$$
substituting into \eqref{calc} and re-arranging all the terms, we  obtain
\begin{equation}
\label{pullb}
\mu_k^* \omega = \sum_{i<j} \hat{b}_{ij} \frac{du_i}{u_i}\wedge \frac{du_j}{u_j}
\end{equation}
with $\hat{B}=[\hat{b}_{ij}]$ given by: 
\begin{itemize}
\item[a)]  for $i<k$ (resp. for $j>k$): $\hat{b}_{ik} = -b_{ik}$ (resp. $\hat{b}_{kj} = -b_{kj}$);
\item[b)] for $i<j<k$:
$$\hat{b}_{ij} = b_{ij} + \left \{ 
\begin{array}{ll} 
\frac{A^+\left ( b_{ik}b_{kj} - b_{jk}b_{ki} \right ) }{A^++A^-} = 0, & \mbox{ if }\, b_{ki}>0, \, b_{kj}>0 \\
& \\
- \frac{A^-\left ( b_{ik}b_{kj} - b_{jk}b_{ki} \right ) }{A^++A^-} = 0, & \mbox{ if }\, b_{ki}<0, \, b_{kj}<0 \\
 & \\
  \frac{A^+ b_{ik}b_{kj} +A^- b_{jk}b_{ki} }{A^++A^-} = b_{ik} b_{jk}, & \mbox{ if }\,  b_{ki}<0, \, b_{kj}>0 \\
 & \\
  - \frac{A^- b_{ik}b_{kj} +A^+ b_{jk}b_{ki} }{A^++A^-} = - b_{ik} b_{jk}, & \mbox{ if } \, b_{ki}>0, \, b_{kj}<0
\end{array}
\right. $$
where the equalities inside the above bracket follow from the skew-symmetry of $B=[b_{ij}]$.

Using similar arguments, the remaining entries $\hat{b}_{ij}$ are as follows.
\item[c)]  for $i<k<j$: 
$$\hat{b}_{ij} = b_{ij} + \left \{ 
\begin{array}{ll} 
\frac{A^+\left ( b_{ik}b_{kj} + b_{kj}b_{ki} \right ) }{A^++A^-} = 0 , & \mbox{ if }\, b_{ki}>0, \, b_{kj}>0 \\
& \\
- \frac{A^-\left ( b_{ik}b_{kj} + b_{kj}b_{ki} \right ) }{A^++A^-}= 0 , & \mbox{ if }\, b_{ki}<0, \, b_{kj}<0 \\
 & \\
  \frac{A^+ b_{ik}b_{kj} - A^- b_{kj}b_{ki} }{A^++A^-} = b_{ik}b_{kj}, & \mbox{ if }\,  b_{ki}<0, \, b_{kj}>0 \\
 & \\
 - \frac{A^- b_{ik}b_{kj} - A^+ b_{kj}b_{ki} }{A^++A^-}=  - b_{ik}b_{kj}, & \mbox{ if } \, b_{ki}>0, \, b_{kj}<0
\end{array}
\right.  $$
\item[d)] for $k<i<j$: 
$$\hat{b}_{ij} = b_{ij} + \left \{ 
\begin{array}{ll} 
\frac{A^+\left ( - b_{ki}b_{kj} + b_{kj}b_{ki} \right ) }{A^++A^-}= 0, & \mbox{ if }\, b_{ki}>0, \, b_{kj}>0 \\
& \\
\frac{A^-\left ( b_{ki}b_{kj} - b_{kj}b_{ki} \right ) }{A^++A^-}=0, & \mbox{ if }\, b_{ki}<0, \, b_{kj}<0 \\
 & \\
 - \frac{A^+ b_{ki}b_{kj} + A^- b_{kj}b_{ki} }{A^++A^-}= b_{ik}b_{kj}, & \mbox{ if }\,  b_{ki}<0, \, b_{kj}>0 \\
 & \\
  \frac{A^- b_{ki}b_{kj} + A^+ b_{kj}b_{ki} }{A^++A^-}=  - b_{ik}b_{kj}, & \mbox{ if } \, b_{ki}>0, \, b_{kj}<0
\end{array}
\right. $$
\end{itemize}
Summing up, if $i=k$ or $j=k$ then $\hat{b}_{ij} = -b_{ij}$, and in any other case
\begin{equation}
\hat{b}_{ij} = b_{ij} + \left \{ 
\begin{array}{ll} 
0, & \mbox{ if }\, b_{ik} b_{kj}<0 \\
 & \\
 b_{ik}b_{kj}, & \mbox{ if }\,  b_{ik}>0, \, b_{kj}>0 \\
 & \\
 -b_{ik}b_{kj}, & \mbox{ if } \, b_{ik}<0, \, b_{kj}<0
\end{array}
\right. \end{equation}

 It is now easy to check that each  $\hat{b}_{ij}$ coincides with $b'_{ij}$ given in \eqref{mut k2}, which shows that $\hat{B}=[\hat{b}_{ij}]$ is precisely $\mu_k(B)$.
\end{proof}

\begin{rem}
Although we do not require compatibility of $\omega$ with the cluster algebra ${\cal A}(B)$, the proof of the above lemma could follow from the proof of Theorem~6.2 in \cite{GeShVa3} which characterizes log presymplectic structures compatible with a cluster algebra. However, to the best of our knowledge the statement in Lema~\ref{lem2}, in particular the expression \eqref{pullmuk},  is new and not completely obvious from the  literature in cluster algebras.
\end{rem}

\bigbreak

\begin{proof}[Proof of Theorem~\ref{presomega}]
Using properties of the pullback operation and the identities \eqref{pullsigma} and \eqref{pullmuk} from Lemmas \ref{lem1} and \ref{lem2},  it follows that
\begin{eqnarray*}
\varphi^* \omega & = & \left ( \sigma^m \circ \mu_m \circ \mu_{m-1}\circ \ldots \circ \mu_1\right )^* \omega =  \mu_1^* \circ \ldots \circ \mu_{m}^*\circ (\sigma^m)^*\omega\\
 & = &   \sum_{i<j}  \left ( \mu_1 (\cdots \mu_{m-1} (\mu_m(\sigma^{-m} B \sigma^{m})))\right )_{ij} \frac{du_i}{u_i}\wedge \frac{du_j}{u_j}.
 \end{eqnarray*}
Therefore $\varphi^*\omega=\omega$ if and only if
$$\mu_1 \circ\cdots\circ \mu_{m-1}\circ \mu_m (\sigma^{-m} B \sigma^{m}) = B.$$

\noindent As any mutation is an involution we have the following equivalence
$$\varphi^*\omega=\omega \quad \Longleftrightarrow \quad \sigma^{-m} B \sigma^{m}= \mu_m\circ\cdots\circ\mu_1 (B),$$
which concludes the proof.
\end{proof}

\section{Symplectic reduction of cluster iteration maps}\label{sec4}

In this section we prove that whenever the matrix $B$ representing an $m$-periodic quiver  is singular, the corresponding iteration map can be reduced to a symplectic map with respect to a {\it log symplectic form}.  

Our main result, Theorem \ref{reducth}, is a generalization for quivers of arbitrary period $m$ of Theorem 2.6 in \cite{FoHo2}. The strategy we follow to prove  this theorem is completely   different  from the one  in \cite{FoHo2}, which  was based on the classification of 1-periodic quivers. In fact our proof relies on a classical theorem of G. Darboux for presymplectic forms and on the invariance of the standard presymplectic form under the iteration map proved in Theorem~\ref{presomega}.

\begin{theorem}
\label{reducth}
Let $Q$ be an $m$-periodic quiver with $N$ nodes, $(B_Q,\mathbf{u})$ the associated initial seed and $\varphi$ the iteration map given in \eqref{itmapm}. If $\rk(B_Q)=2k<N$ then there exist 
\begin{enumerate}
\item[i)] a submersion $\pi:  \Rb_+ ^N \longrightarrow   \Rb_+^{2k}$,
\item[ii)] a map $\hat{\varphi}: \Rb_+ ^{2k} \longrightarrow  \Rb_+ ^{2k}$
\end{enumerate}
such that 
the following diagram is commutative
$$\xymatrix{ \Rb_+ ^N  \ar @{->} [r]^\varphi\ar @{->} [d]_{\pi} & \Rb_+ ^N \ar @{->}[d]^{\pi}\\
 \Rb_+ ^{2k}  \ar @{->} [r]_{\hat{\varphi}} &  \Rb_+ ^{2k}}$$
Furthermore $\hat{\varphi}$ is symplectic with respect to  the \emph{canonical log symplectic form}
\begin{equation}
\label{omega0}
\omega_0=\frac{dy_1}{y_1}\wedge \frac{dy_2}{y_2} + \cdots + \frac{dy_{2k-1}}{y_{2k-1}}\wedge \frac{dy_{2k}}{y_{2k}},
\end{equation}
that is $\hat{\varphi}^* \omega_0=\omega_0$.
\end{theorem}

\bigbreak

The proof of this  theorem relies on the next proposition which in turn relies on Darboux's theorem for closed 2-forms (or presymplectic forms)  of constant rank. 

\begin{proposition}
\label{reduc}
Let $\omega$ be a closed 2-form of constant rank $2k$ on a manifold $M$ and  $x_0\in M$. Then, there exists a set $\{g_1,g_2, \ldots , g_{2k}\}$ of $2k$ locally independent functions on $M$ such that
\begin{equation}
\label{cart}
\omega = dg_1 \wedge dg_2 + \cdots + dg_{2k-1}\wedge dg_{2k}.
\end{equation}
Moreover, if $\varphi$ is a local diffeomorphism preserving $\omega$, that is  $\varphi^*\omega=\omega$, then each of the functions 
$$\psi_1= g_1\circ \varphi, \quad \psi_2 = g_2\circ \varphi, \quad \ldots  \quad , \psi_{2k}= g_{2k}\circ \varphi$$
depends only on $\{ g_1, g_2, \ldots , g_{2k}\}$.

\end{proposition}

\begin{proof} The first statement is just Darboux's theorem for closed 2-forms of constant rank (see for instance \cite{AbMa} and \cite{SS}).\

To prove that the functions $\,\, \psi_i= g_i\circ \varphi\,\,$ just depend on the set $\{ g_1, g_2, \ldots , g_{2k}\}$, let $\omega^{(k)}$ be the $k^{th}$ exterior power of $\omega$. As $\varphi^*\omega = \omega$ then  
\begin{equation}
\label{inv}
\varphi^*\omega^{(k)} = \omega^{(k)}.
\end{equation}
Using the expression \eqref{cart} for $\omega$, it turns out that
$$\omega^{(k)} =(-1)^{[k/2]}k! \, dg_1\wedge dg_2 \wedge \ldots \wedge dg_{2k-1}\wedge dg_{2k},$$
and so  \eqref{inv} is equivalent to
\begin{equation}
\label{dfi}
d\psi_1\wedge d\psi_2 \wedge \ldots \wedge d\psi_{2k-1}\wedge d\psi_{2k} = dg_1\wedge dg_2 \wedge \ldots \wedge dg_{2k-1}\wedge dg_{2k}.
\end{equation}
Now complete the set $\{ g_1, g_2, \ldots ,g_{2k}\}$ to a full set of coordinates $\{ g_1, g_2, \ldots , g_N\}$ on $M$, and consider the Jacobian matrix of $\psi = ( \psi_1, \psi_2, \ldots , \psi_{2k})$,
\begin{equation}
\nonumber
J = \begin{bmatrix}
\pder{\psi_1}{g_1} &  \cdots & \pder{\psi_1}{g_{2k}}  & \pder{\psi_1}{g_{2k+1}} &  \cdots & \pder{\psi_1}{g_N}  \\
\\
\vdots &  & \vdots & \vdots & & \vdots \\
\\
\pder{\psi_{2k}}{g_1} &  \cdots & \pder{\psi_{2k}}{g_{2k}}  & \pder{\psi_{2k}}{g_{2k+1}} &  \cdots & \pder{\psi_{2k}}{g_N}  
\end{bmatrix}.
\end{equation}
Condition \eqref{dfi} implies that the determinant of the leftmost $(2k)\times(2k)$ submatrix of $J$ is equal to 1 and all the other $(2k)\times(2k)$ determinants of $J$ are equal to 0.

Linear algebra arguments assure that the rightmost $(2k)\times(N-2k)$ submatrix of $J$ is the zero matrix, which concludes the proof.
\end{proof}
\bigbreak

\begin{proof}[Proof of Theorem \ref{reducth}] Let $\omega$ be the standard log presymplectic form on $\Rb_+^N$
$$\omega = \sum_{1\leq i<j\leq N} b_{ij} \frac{du_i}{u_i} \wedge \frac{du_j}{u_j}.
$$
Since $\omega$ has rank $2k$, Darboux's theorem (see Proposition~\ref{reduc}) implies the existence of functions $g_1, \ldots, g_{2k}$ such that $\omega$ is given by \eqref{cart}. Considering
$$\begin{array}{rcl}
\pi :  \Rb_+ ^N & \longrightarrow & \Rb_+ ^{2k}\\
\mathbf{u} & \longmapsto & \left ( \exp(g_1(\mathbf{u})), \ldots, \exp(g_{2k}(\mathbf{u}))\right),
\end{array}$$ 
and $\omega_0$ given by \eqref{omega0}, we have 
\begin{equation}
\label{piomega}
\pi^*\omega_0 = \omega.
\end{equation}

By Theorem \ref{presomega} the iteration map $\varphi$ preserves $\omega$, that is  $\varphi^*\omega = \omega$.  Then,  again by Proposition~\ref{reduc},
$\pi\circ\varphi (\mathbf{u})$ depends only on $\pi(\mathbf{u})$, since each $\,\,g_i\circ \varphi\,\,$ depends only on $\{ g_1, \ldots, g_{2k}\}$.
This is equivalent to say that  $\hat{\varphi}$ exists and makes the diagram commutative.

It remains to prove that $\hat{\varphi}$ is symplectic. For this purpose, we note that the  commutativity of the diagram is equivalent to
$$( \pi\circ \varphi)^*\omega_0 = (  \hat{\varphi}\circ \pi)^*\omega_0 \quad \Longleftrightarrow \quad \varphi^*(\pi^*\omega_0)= \pi^* ( \hat{\varphi}^*\omega_0).$$
Using \eqref{piomega} and the fact that  $\varphi$ preserves $\omega$,  we have
$$\varphi^*(\pi^*\omega_0)= \pi^* ( \hat{\varphi}^*\omega_0) \quad \Longleftrightarrow \quad \pi^*\left (\omega_0  -  \hat{\varphi}^* \omega_0  \right ) = 0.$$
As $\pi$ is a submersion, this implies $\hat{\varphi}^* \omega_0 = \omega_0$. 
\end{proof}

\begin{rem}
\label{rem2}
In the proof of Theorem \ref{reducth}, we can use any (nonzero) multiple of $\omega$. In fact, if $B$ comes from an $m$-periodic quiver then $\lambda\omega$ is also preserved by $\varphi$ (direct consequence of Theorem \ref{presomega}), Proposition 1 can still be used with $\lambda\omega$, and reduction is achieved in an entirely analogous way. 
 This fact will be used in Example \ref{ex1} to improve the form of the reduced iteration map. 
\end{rem}

We note that without the assumption of periodicity of the quiver $Q$ (that is, without  the iteration map $\varphi$), Theorem~\ref{reducth}  may be seen as an alternative proof of the existence of the secondary cluster manifold with its Weil-Petersson form (see for instance \cite{GeShVa3}). However, when taking into account the iteration map,  Theorem~\ref{reducth} has stronger consequences since it allows to reduce recurrence relations to the secondary manifold.

\section{Computation of the reduced iteration map}\label{sec5}

In this section we compute the reduced iteration map $\hat{\varphi}$ as well as the reduced variables $\pi(\mathbf{u})$ 
for three periodic quivers. To the best of our knowledge this reduction has not been performed before for quivers of period higher than 1.

Our approach to the computation of the reduced symplectic iteration map $\hat{\varphi}$ is based on the construction of the functions $g_1, \ldots, g_{2k}$ appearing in Proposition~\ref{reduc}. We note that these functions are given by Darboux's theorem  whose proof is not constructive. However due to the particular form of the log presymplectic forms it is possible to change coordinates in such a way that one can apply a theorem of Cartan (Theorem~2.3  in  \cite{LiMa}) whose proof is constructive,  and then obtain explicit reduced variables $f_i=\exp(g_i(\mathbf{u}))$.

Let us consider the  standard log presymplectic form $\omega$ written in coordinates $v_i = \log u_i$,
$$\omega =  \sum_{1\leq i<j\leq N} b_{ij} dv_i \wedge dv_j,$$
and assume that its rank is $2k<N$.

Cartan's Theorem says that there exist $2k$ functions $g_i$ (depending  linearly on the $v_i$ variables)  such that  
$$\omega =  \sum_{1\leq i<j\leq 2k} dg_i \wedge dg_j. $$
We recall the main steps of the proof of Cartan's Theorem, which explicitly produces the functions $g_i$.
\bigbreak

\noindent {\it 
Reordering if necessary the $v_i$-coordinates, we can assume that $b_{12}\neq 0$. Let
\begin{equation}
\label{cartg}
g_1= \frac{1}{b_{12}} \sum_{k=1}^N b_{1k} v_k \quad \mbox{and} \quad g_2=\sum_{k=1}^N b_{2k} v_k,
\end{equation}
so that 
$$dg_1\wedge dg_2 = b_{12} dv_1\wedge dv_2 + \sum_{i=3}^N b_{1i} dv_1\wedge dv_i +  \sum_{j=3}^N b_{2j} dv_2\wedge dv_j+ \alpha$$
where $\alpha$ depends only on $\{ v_3, \ldots , v_N\}$. Then the 2-form
$$\tilde{\omega}=\omega - dg_1\wedge dg_2$$
is a closed 2-form on the $(N-2)$-dimensional vector space, with coordinates $\{ v_3, \ldots , v_N\}$, and $\rk (\tilde{\omega})=2k-2$.

If $\rk (\omega)=2$ then $\tilde{\omega}= 0$ and $\omega = dg_1\wedge dg_2.$ Otherwise, the previous procedure is repeated, replacing $\omega$ by $\tilde{\omega}$. After $k$ steps all the functions $g_1, \ldots, g_{2k}$ will have been obtained.}

\bigbreak
As each function $g_i$ is a linear function of the  variables $v_i=\log u_i$,  it can be written in the form
$$g_i (u_1,\ldots ,u_N) = \log ( f_i(u_1,\ldots, u_N)), \quad i=1, 2, \ldots N.$$
The submersion $\pi$ in Theorem \ref{reducth} is then given by 
$$\pi (u_1,\ldots ,u_N) = \left ( f_1(u),\ldots ,f_{2k}(u)\right ),$$
and the functions $f_1, \ldots, f_{2k}$ are reduced variables. Again by Theorem~\ref{reducth},  $\hat{\varphi}\circ\pi = \pi\circ\varphi$ and so the reduced iteration map is
$$\hat{\varphi} (f_1,\ldots ,f_{2k}) = ( f_1 \circ \varphi , \ldots, f_{2k} \circ \varphi ).$$

\bigbreak

In the next three examples we  consider  quivers represented by the matrix $B$ in \eqref{matrix62per} and we compute the respective reduced iteration map. In the first two examples the quivers are  2-periodic and the computation of the reduced iteration map follows step by step  the described proof of Cartan's theorem. The last example is a 1-periodic quiver  and aims to  illustrate the possibility of modifying  Cartan's construction to choose a more convenient set of reduced variables.

Recall that $B$ in \eqref{matrix62per} represents a 1-periodic quiver if $r=t$ and a 2-periodic quiver otherwise. Furthermore, if $p=r+t$ then $B$ fails to have maximal rank and so the iteration map $\varphi$ can be reduced to a symplectic map (by Theorem~\ref{reducth}).

\begin{example} 
\label{ex2}
 {\it A 2-periodic quiver $Q$ with 6 nodes and $\operatorname{rank} B_Q=2$.}  
\bigbreak 
 
When  $r=2, t=5,p=7,s=13$ the matrix $B$ in   \eqref{matrix62per}  has $\operatorname{rank}$ 2.
 From \eqref{itmap4nos2pernova2} the iteration map is
$$\varphi(u_1,u_2,u_3,u_4, u_5,u_6)= (u_3,u_4,u_5,u_6,u_7,u_8)
$$
with
\begin{equation}\label{u7u8rank2}
u_7=\frac{u_3^{13}u_5^{13}+u_2^2u_4^7u_6^5}{u_1}, \qquad u_8=\frac{u_4^{13}u_6^{13}+u_3^{5}u_5^{7}u_7^2}{u_2}.
\end{equation}
The first and second rows of $B$ are respectively  
$$(0,-2,13,-7,13,-5) \quad \mbox{and}  \quad (2,0,-31,13,-33,13),$$
and so \eqref{cartg} gives 
 $\omega= dg_1\wedge dg_2,$
with
$$g_1=v_2-\frac{13}{2}v_3+\frac{7}{2}v_4-\frac{13}{2}v_5+\frac{5}{2}v_6 =\log \left ( \frac{u_2u_4^{7/2}u_6^{5/2}}{u_3^{13/2}u_5^{13/2}}\right )$$
and 
$$g_2=2v_1-31v_3+13v_4-33v_5+13v_6=\log \left ( \frac{u_1^2u_4^{13}u_6^{13}}{u_3^{31}u_5^{33}}\right ).$$
The reduced log symplectic form is $\hat{\omega}= \frac{df_1}{f_1}\wedge \frac{df_2}{f_2}$,
with
$$f_1=\frac{u_2u_4^{7/2}u_6^{5/2}}{u_3^{13/2}u_5^{13/2}}, \qquad f_2=\frac{u_1^2u_4^{13}u_6^{13}}{u_3^{31}u_5^{33}}.$$
To obtain the reduced  iteration map $\hat{\varphi}= (\hat{\varphi}_1, \hat{\varphi}_2)$ we just have to compute  the compositions 
$f_i\circ \varphi$ as functions of $f_1,f_2$, which produces
\begin{align*}
\hat{\varphi}_1=f_1\circ \varphi& = \frac{f_2^{3/4}\left(f_2+ (1+f_1^2)^2\right)^{5/2}}{f_1^{5/2} (1+f_1^2)^{13/2} },\\
\hat{\varphi}_2=f_2 \circ\varphi& =\frac{f_2^{7/2} \left(f_2+ (1+f_1^2)^2\right)^{13}}{f_1^{13} (1+f_1^2)^{33}}.
\end{align*}
\end{example}

\begin{example}
\label{ex21}
{\it A 2-periodic quiver $Q$ with 6 nodes and $\operatorname{rank} B_Q=4$.}
\bigbreak

Taking $r=1, s=1, t=2,p=3$ in  \eqref{matrix62per}, we have   $\operatorname{rank} B= 4$. In this case,  one needs to go one step further in the proof of Cartan's Theorem in order to compute the reduced functions $f_i$. 

The iteration map is
$\varphi(u_1,u_2,u_3,u_4, u_5,u_6)= (u_3,u_4,u_5,u_6,u_7,u_8)
$
with
\begin{equation}\label{u7u8rank22}
u_7=\frac{u_3u_5+u_2u_4^3u_6^2}{u_1}, \qquad u_8=\frac{u_4u_6+u_3^2u_5^3u_7}{u_2}.
\end{equation}

Using  Cartan's construction we get  $\omega= dg_1\wedge dg_2+ dg_3\wedge dg_4$ with
$$dg_1\wedge dg_2 = d\left(v_2-v_3+3v_4-v_5+2 v_6\right)\wedge d\left(v_1-3v_3+v_4-4v_5+v_6\right),$$
and 
$$dg_3\wedge dg_4 = d(v_4+v_6)\wedge d(8v_3+8v_5).$$
From the above expressions,  we obtain 
\begin{equation}\label{invarank22}
f_1=\frac{u_2u_4^3u_6^2}{u_3u_5},\quad f_2=\frac{u_1u_4u_6}{u_3^3u_5^4},\quad
f_3= u_4 u_6,\quad f_4=u_3^8u_5^8.
\end{equation}
The reduced iteration map is $\hat{\varphi}= (\hat{\varphi}_1,\hat{\varphi}_2,\hat{\varphi}_3,\hat{\varphi}_4)$ with
\begin{align*}
\hat{\varphi}_1=f_1\circ \varphi&=\frac{f_3^8 (1+f_1+f_2)^2}{f_1^2f_2 (1+f_1)}\\
\hat{\varphi}_2=f_2\circ\varphi&=\frac{f_2^3f_4 (1+f_1+f_2)}{f_1 (1+f_1)^4}\\
\hat{\varphi}_3=f_3\circ\varphi&=\frac{f_3^4(1+f_1+f_2)}{f_1f_2f_4^{1/8}}\\
\hat{\varphi}_4=f_4\circ\varphi&=\frac{f_3^8 (1+f_1)^8}{f_2^{8} f_4^{2}}.
\end{align*}
\end{example}

\begin{example} {\it A 1-periodic quiver with 6 nodes and $\operatorname{rank} B_Q=2$.}
\label{ex1}
\bigbreak

Let us consider the matrix $B$ in \eqref{matrix62per} with $r=t=2$, $p=4$ and $s=6$, which represents a 1-periodic quiver. The rank of $B$ is 2 and its first and second rows are, respectively, $(0,-2,6,-4,6,-2)$ and $(2,0,-14,6,-16,6)$. 

The iteration map is  
$\varphi(u_1,u_2,u_3,u_4, u_5,u_6)= (u_2,u_3,u_4,u_5,u_6,u_7)
$
with
\begin{equation}\label{1per6nodesred}
u_7=\frac{u_2^2u_4^4u_6^2+u_3^6u_5^6}{u_1}.
\end{equation}

The reduction of this iteration map was done in \cite{InEs} and also in \cite{FoHo} (using a different approach). We include it here in order to illustrate the possibility of using Cartan's proof  to obtain a more convenient set of reduced variables. The strategy of the previous examples is modified as follows. 

We start by considering the  presymplectic form $\omega' = -\frac{1}{2} \omega$ (see Remark~\ref{rem2}). That is,
$$\omega'= d\underbrace{\left(v_2-3v_3+ 2v_4-3v_5+v_6\right)}_{g_1'}\wedge d\underbrace{\left(-v_1+7v_3-3v_4+8v_5-3v_6\right)}_{g_2'}.
$$
Next, noting that $\omega'= d g_1'\wedge d g_2' = d g_1'\wedge d (g_2' +3 g_1')$, and taking $g_1= -g_2'-3g_1'$ and $g_2=g_1'$ we obtain
$$\omega'= dg_1\wedge dg_2= d\left(v_1-3v_2+2v_3-3v_4+v_5\right)\wedge d\left(v_2-3v_3+ 2v_4-3v_5+v_6\right).$$

The reduced variables are then 
$$f_1=\frac{u_1u_3^2u_5}{u_2^3u_4^3}, \quad f_2=\frac{u_2u_4^2u_6}{u_3^3u_5^3},$$
and the reduced  map $\hat{\varphi}$, which is  symplectic with respect to the form
$\hat{\omega} =  \frac{df_1}{f_1}\wedge\frac{df_2}{f_2},$
is given by 
$$\hat{\varphi} (f_1,f_2)=\left(f_2, \frac{1+f_2^2}{f_1f_2^3}\right).$$
\end{example}

 Therefore the sixth order recurrence relation \eqref{rec6nodes1p}
$$
u_{n+6} u_n= u_{n+1}^2u_{n+3}^2u_{n+5}^4+u_{n+2}^6u_{n+4}^6, \quad n=1,2,\ldots
$$
reduces to the following relation
$$f_{n+2} = \frac{1+f_{n+1}^2}{f_nf_{n+1}^3}, \quad n=1,2,\ldots $$
which is a recurrence relation of order 2.

\small{}

\end{document}